\documentclass[11pt]{article}
\usepackage{amssymb,amsmath}
\usepackage[pdftex]{color}
\usepackage{amsthm}
\usepackage{amsmath}
\usepackage{amssymb}
\usepackage{graphicx}
\usepackage{amsfonts}
\usepackage{esint}

\usepackage{geometry}
\numberwithin{equation}{section}
\numberwithin{figure}{section}

\newtheorem{theorem}{Theorem}
\newtheorem{proposition}{Proposition}

\newtheorem{remark}{Remark}
 
\newcommand{\Rn}{\mathbb{R}^n}
\newcommand{\R}{\mathbb{R}}
\newcommand{\Om}{\Omega}
\newcommand{\pt}{\partial}

\numberwithin{equation}{section}
\numberwithin{theorem}{section}
\numberwithin{proposition}{section}
\numberwithin{definition}{section}
\numberwithin{corollary}{section}
\numberwithin{lemma}{section}
\numberwithin{example}{section}
\usepackage{lipsum}

\begin{document}
{\parindent0pt
%
%
%
%
%
%
%
%
%
        \title{On Some Nonlocal in Time and Space Parabolic Problem}
        \author{Sandra Carillo\footnote{Dipartimento Scienze di Base e Applicate
            per l'Ingegneria  ``\textsc{La Sapienza}''   Universit\`a di Roma,  16, Via A. Scarpa, 00161 Rome, Italy}
        ~~{}\footnote{Gr. Roma1, IV - Mathematical Methods in NonLinear Physics, National Institute for Nuclear Physics (I.N.F.N.), Rome, Italy} and Michel Chipot  \footnote{Institute of   Mathematics, University of Z\"urich, Winterthurerstr.190, CH-8057 Z\" urich, email : m.m.chipot@math.uzh.ch} ~{}\footnote{FernUni Schweiz, Schinerstrasse 18, CH-3900 Brig-Glis}}
        \date{ }
        \maketitle
\begin{abstract} The goal of this note is to study  nonlinear parabolic problems nonlocal in time and space. 
We first establish the existence of a solution and its uniqueness in certain cases. Finally we consider its asymptotic behaviour.
\end{abstract}

\noindent
{\bf MSC2020-Mathematics Subject Classification: 35A01, 35D30, 35F25, 35K10}.\\[0.2cm]


\


\section{Introduction and notation} 

We will denote by $\Om$ a smooth bounded open set of $\Rn$, $n\geq 2$ with boundary $\partial \Om$. We would like to consider the following problem. 
Find $u = u(x,t)$ such that 
 \begin{equation}\label{1.1}
  \begin{cases} u_t -  \alpha \big (\int_\Om g(x)u(x,t) dx \big )\Delta u  + \beta \big (\int_0^t h(s)u(x,s) ds \big )u = f ~~{\text in }~ \Om \times (0,T), \cr
  u(\cdot,t) = 0 ~{\text on }~ \partial \Om,~ t\in (0,T), ~u(x,0) = u_0(x).
  \end{cases}
 \end{equation}
$T$ is a positive number, $u_0, f, g, h$ are given data. The equation could be regarded as a model of population dynamics where $u(x,t)$ is the density of a 
population at the location $x$, at the time $t$. The nonlinear terms are an account for a death or diffusion  rate which at time $t$ depends on the total population having been at the location $x$ in the past or in $\Om$ at time $t$.  To study this issue we were inspired by the papers  \cite{S1}-\cite{Walker} where a similar problem was introduced at the difference that in \eqref{1.1} the integral goes 
up to $T$ which we think is an interesting point of view but perhaps a bit surprising from a realistic one in our framework.

The paper is divided as follows. In the next section we prove existence of a weak solution to \eqref{1.1}. In the subsequent section we establish a result of uniqueness. 
Note that in comparison to \cite{S1}-\cite{Walker}  our result is global. Finally we study in a simple case the asymptotic behaviour of the solution to \eqref{1.1}.

\section{A  result of existence} 
\vskip .4 cm 
We denote by $L^p(\Om), 1 \leq p \leq  +\infty$ the usual $L^p$-space on $\Om$. It is equipped with its usual norm and for instance, in the case where $p=2$, we denote it by $|~|_2$ i.e. 
\begin{equation*}
|v|_2^2 = \int_\Om v(x)^2 dx ~~ \forall v \in L^2(\Om).
\end{equation*}
We refer the reader to \cite{CNew}-\cite{Evans} for the notation used in the sequel, for instance for $H_0^1(\Om)$ or its dual $H^{-1}(\Om)$ or the spaces $L^2(0,T;V)$, $L^2(0,T;V')$ when $V$ is a Banach space.
\vskip .3 cm

The main result of this section is the following :
\begin{theorem}\label{2.1} Set $V= H_0^1(\Om), V'= H^{-1}(\Om)$. Suppose 
 \begin{equation}\label{2.1}
 u_0 \in L^2(\Om), ~f \in L^2(0,T;V'), ~\alpha, \beta \in C(\R)\cap L^\infty(\R), ~g\in  L^\infty(\Om), h \in  L^\infty(0,T), ~\forall T,
 \end{equation}
 and that for some positive constant $a$ one has
  \begin{equation}\label{2.2}
0< a \leq \alpha.
  \end{equation}

Then there exists a weak solution to \eqref{1.1}. $C(\R)$ denotes the space of continuous functions.
\end{theorem}
\begin{proof} 1. One can assume that $\beta \geq 1$.

Indeed, suppose that we can solve \eqref{1.1} in this case. $u$ is solution to \eqref{1.1} iff
\begin{equation}\label{2.3}
{\tilde u}= e^{-\lambda t }u
 \end{equation}
satisfies 
\begin{equation*}
\begin{aligned}
&(e^{\lambda t }{\tilde u})_t - \alpha \big (\int_\Om g(x)e^{\lambda t }{\tilde u}(x,t) dx\big )e^{\lambda t }\Delta {\tilde u} + \beta \big (\int_0^t h(s)e^{\lambda s }{\tilde u}(x,s) ds \big )e^{\lambda t }{\tilde u}= f, \cr
\Leftrightarrow ~~& e^{\lambda t }{\tilde u}_t + \lambda e^{\lambda t }{\tilde u}-\alpha \big (\int_\Om g(x)e^{\lambda t }{\tilde u}(x,t) dx\big )e^{\lambda t }\Delta {\tilde u} + \beta \big (\int_0^t h(s)e^{\lambda s }{\tilde u}(x,s) ds \big )e^{\lambda t }{\tilde u} = f, \cr
\Leftrightarrow ~~& {\tilde u}_t  -\alpha \big (\int_\Om g(x)e^{\lambda t }{\tilde u}(x,t) dx\big )\Delta {\tilde u} +\big\{\lambda+ \beta \big (\int_0^t h(s)e^{\lambda s }{\tilde u}(x,s) ds\big) \big\}{\tilde u} =e^{-\lambda t } f, \cr
\end{aligned}
 \end{equation*}
i.e. iff ${\tilde u}$ satisfies \eqref{1.1} with $f,h, \alpha$ replaced respectively by $e^{-\lambda t }f,e^{\lambda t}h, \alpha(e^{\lambda t} .)$ and $\beta$ by $\lambda+\beta$ which is greater than 1 for $\lambda$ large enough.
\vskip .3 cm
2. We suppose that $\beta \geq 1$.
\vskip .2 cm

\noindent Let $w \in L^2(0,T;L^2(\Om)) \subset  L^1(0,T;L^1(\Om))$. Then, see \cite{CNew}, there exists a unique  $u=S(w)$ solution to
\begin{equation}\label{2.4}
\begin{cases} u\in L^2(0,T;V),~ u_t \in L^2(0,T;V'), \cr
  \frac{d}{dt}(u,v) + &\hskip -4.5 cm \alpha \big (\int_\Om g(x)w(x,t) dx \big )\int_\Om\nabla u \cdot \nabla v dx  + \big(\beta \big (\int_0^t h(s)w(x,s) ds \big )u,v\big)\cr 
  & ~~~~~~~~~~~~= \langle f,v \rangle ~~\forall v \in H_0^1(\Om), \text{ in } {\cal{D}'}(0,T).\cr
   \end{cases}
 \end{equation}
In the equation above we denote by $(~,~)$ the canonical scalar product in $L^2(\Om)$ and by $\langle~,~\rangle$ the duality between $H^{-1}(\Om)$ and $ H_0^1(\Om)$, ${\cal{D}}(0,T)$ and 
${\cal{D}'}(0,T)$ denote respectively the space of $C^\infty$ functions with compact support in $(0,T)$ and its dual, the usual space of distributions 
on $(0,T)$. (Cf. for instance \cite{CNew}). We will be done if we can show that $S$ has a fixed point. Taking in the equation above $v=u$ we get easily if $a\wedge1$ denotes the minimum of $a$ and 1
\begin{equation*}
\begin{aligned}
\frac{1}{2}\frac{d}{dt}|u|_2^2 + a\wedge1\int_\Om (\nabla u \cdot \nabla u + u^2) dx &\leq \langle f,u \rangle \leq |f|_{V'} \big| |\nabla u| \big|_2 \cr 
&\leq \frac{1}{2(a\wedge1)}|f|^2_{V'}+ \frac{a\wedge1}{2} \big| |\nabla u| \big|^2_2.\cr
\end{aligned}
 \end{equation*}
$|f|_{V'}$ denotes the strong dual norm of $f$ in $H^{-1}(\Om)$ associated to the norm $\big| |\nabla u| \big|_2$ in $ H_0^1(\Om)$. From this we derive 
\begin{equation*}
\frac{d}{dt}|u|_2^2 + (a\wedge1)\int_\Om (\nabla u \cdot \nabla u + u^2) dx \leq \frac{1}{a\wedge1}|f|^2_{V'}
 \end{equation*}
and after an integration in $t$
\begin{equation*}
|u|_2^2 +(a\wedge1 )\int_0^t\int_\Om (\nabla u \cdot \nabla u + u^2) dx ds \leq |u_0|^2_2 + \frac{1}{a\wedge1}\int_0^t |f(\cdot,s)|^2_{V'}ds.
 \end{equation*}
It follows that
\begin{equation*}
|u|_{L^2(0,T;V)}^2,~ |u|_{L^2(0,T;L^2(\Om))}^2 \leq C^2 =\frac{1}{(a\wedge1 )} \Big(|u_0|^2_2 + \frac{1}{a\wedge1}\int_0^T |f(\cdot,s)|^2_{V'}ds\Big).
 \end{equation*}
Set 
\begin{equation*}
B = \{v\in L^2(0,T;L^2(\Om)) ~|~|v|_{L^2(0,T;L^2(\Om))}\leq C \}.
 \end{equation*}
Clearly, $S$ maps $B$ into itself. Moreover since 
\begin{equation*}
u_t = \alpha \big (\int_\Om g(x)w(x,t) dx \big )\Delta u - \beta \big (\int_0^t h(s)w(x,s) ds \big )u + f ~\text{ in } V'
 \end{equation*}
$u_t$ is uniformly bounded in $L^2(0,T;V')$ and $S(B)$ is relatively compact in $B$. The existence of a weak solution to \eqref{1.1} will follow by the Schauder fixed point theorem if $S$ is continuous. To show that, let $w_n \in L^2(0,T;L^2(\Om))$ such that 
\begin{equation*}
w_n \to w \text{ in } L^2(0,T;L^2(\Om)).
 \end{equation*}
Denote  $u_n = S(w_n)$. The estimates above hold and one can extract a subsequence such that, if we still label it by $n$
\begin{equation}
\begin{aligned}
& gw_n \to g w ~\text{ in }~ L^2(0,T;L^2(\Om)), \cr
& hw_n \to h w ~\text{ in }~ L^2(0,T;L^2(\Om)), \cr
& u_n  \to u_\infty  ~\text{ in }~ L^2(0,T;L^2(\Om)),\cr
& \nabla u_ n \rightharpoonup  \nabla u_\infty ~\text{ in }~ L^2(0,T;L^2(\Om)),\cr
& (u_n)_t  \rightharpoonup (u_\infty)_t ~\text{ in }~ L^2(0,T;V'). \cr
\end{aligned}
 \end{equation}
By definition of $u_n$ we have for every $v\in H_0^1(\Om)$ and every $\varphi \in {\cal{D}}(0,T)$
\begin{equation*}
\begin{aligned} \int_0^T-(u_n,v)\varphi'(t) dt + & \int_0^T\varphi(t)\int_\Om\alpha \big (\int_\Om g(x)w(x,t) dx\big ) \nabla u_n \cdot \nabla v ~dxdt \cr &+  \int_0^T\varphi(t)\int_\Om \beta \big (\int_0^t h(s)w_n(x,s) ds \big )u_n v ~dxdt=  \int_0^T \langle f,v \rangle \varphi(t) dt.\cr
\end{aligned}
 \end{equation*}
By the Lebesgue theorem 
\begin{equation}\label{2.5}
\varphi(t) \beta \big (\int_0^t h(s)w_n(x,s) ds \big) v \to \varphi(t) \beta \big (\int_0^t h(s)w(x,s) ds \big) v~~\text{ in} ~ L^2(0,T;L^2(\Om)).
 \end{equation}
Indeed, note that 
\begin{equation*}
\begin{aligned}
|\int_0^t h(s)w_n(x,s) ds -& \int_0^t h(s)w(x,s) ds| \cr
&\leq \int_0^T |h|_\infty|w_n-w|(x,s)ds \cr
&~~~~~~~~~~~~~~~~\leq |h|_\infty \sqrt T \{ \int_0^T (w_n-w)^2(x,s)ds\}^\frac{1}{2}\to 0 ~~{\text a.e.}
\end{aligned}
 \end{equation*}
up to a subsequence. $|h|_\infty$ is the $L^\infty(0,T)$-norm of $h$. Using \eqref{2.5} and the analogue written for $\alpha$ and $g$, one can pass to the limit in the equation satisfied by $u_n$. It follows that $u_\infty=S(w)$. Since the limit of $u_n$ is unique the whole sequence $u_n$ converges toward $u_\infty= S(w)$ and thus $S$ is continuous. This completes the proof of the theorem.
\end{proof}
\begin{remark} The same existence result holds if in \eqref{1.1} one replaces the integral on $(0,t)$ by
\begin{equation*}
\int_0^{t'} h(s)u(x,s) ds
 \end{equation*}
 where $t'$ is any real number in $(0,T]$.
\end{remark}
\section{Uniqueness issue} 
\vskip .4 cm  
One has the following estimate for the solution to \eqref{1.1}:
\begin{proposition} Suppose that $u_0 \in L^\infty(\Om)$, $f \in L^\infty(\Om \times (0,T))$, $\beta \geq 1$. Then it holds 
\begin{equation}\label{3.1}
|u| \leq K = |f|_\infty\vee |u_0|_\infty.
\end{equation}
($\vee$ stands for the maximum of two numbers).
\end{proposition}
\begin{proof} One has 
\begin{equation*}
\frac{d}{dt}(u-K) - \nabla \cdot \Big( \alpha \big(\int_\Om gu~dx\big)\nabla (u-K)\Big)+ \beta \big(\int_0^t hu~ds\big) u - K = f -K \leq 0.
\end{equation*}
It follows, using as test function $(u-K)^+$
\begin{equation*}
\frac{1}{2}\frac{d}{dt}|(u-K)^+|_2^2 + a\wedge1\int_\Om |\nabla  (u-K)^+|^2 + ((u-K)^+)^2 \leq 0.
\end{equation*}
This implies that 
\begin{equation*}
\frac{d}{dt}\big(|(u-K)^+|_2^2 e^{2(a\wedge1)t} \big)\leq 0
\end{equation*}
and since this quantity vanishes at $0$ it vanishes for all time. This shows that $u\leq K$. Since $-u$ satisfies a similar equation one has also $-u\leq K$. This completes the proof of the proposition.
\end{proof}

One can then prove the following uniqueness result :
\begin{theorem} Suppose that $u_0 \in L^\infty(\Om)$, $f \in L^\infty(\Om \times (0,T))$, $g\in L^\infty(\Om),$ $h \in L^\infty(0.T)$. Suppose that $\beta\geq 1$, $\alpha$ are Lipschitz continuous in the sense that for some positive constant $C_\alpha$, $C_\beta$
\begin{equation}
|\alpha(\xi)- \alpha(\eta)| \leq C_\alpha|\xi-\eta|,~~~ |\beta(\xi)- \beta(\eta)| \leq C_\beta|\xi-\eta| ~~\forall \xi, \eta \in \R,
\end{equation}
then the weak solution to \eqref{1.1} is unique.
\end{theorem}
\begin{proof}
Let $u_1, u_2$ be two solutions to \eqref{1.1}. By subtraction one gets 
\begin{equation*}
\begin{aligned}
\frac{d}{dt} (u_1-u_2) &-  \alpha \big (\int_\Om g(x)u_1(x,t) dx \big )\Delta (u_1-u_2) + \beta\big(\int_0^t h(s)u_1(x,s)ds \big) (u_1-u_2) \cr
&= \Big(\alpha \big (\int_\Om g(x)u_1(x,t) dx \big )-  \alpha \big( \int_\Om g(x)u_2(x,t) dx \big)\Big)\Delta u_2 \cr
&~~~~~~~~~~~~ -  \{\beta\big(\int_0^t h(s)u_1(x,s)ds\big) - \beta\big(\int_0^t h(s)u_2(x,s)ds\big) \}u_2.\\
\end{aligned}
\end{equation*}
Multiplying by $(u_1-u_2)$ and integrating on $\Om$ we get 
\begin{equation*}
\begin{aligned}
\frac{1}{2}\frac{d}{dt} |u_1-u_2|_2^2 &+  \alpha \big (\int_\Om g(x)u_1(x,t) dx \big )\int_\Om |\nabla (u_1-u_2)|^2dx  \cr &
~~~~~~~~~~~~~~~~~~~+\int_\Om \beta\big(\int_0^t h(s)u_1(x,s)ds \big) (u_1-u_2)^2dx \cr
&= -\int_\Om \Big(\alpha \big (\int_\Om g(x)u_1(x,t) dx \big )-  \alpha \big( \int_\Om g(x)u_2(x,t) dx \big)\Big)\nabla u_2\cdot \nabla (u_1-u_2)dx  \cr
&~~~~~~~~~~~~ -  \int_\Om \{\beta\big(\int_0^t h(s)u_1(x,s)ds\big) - \beta\big(\int_0^t h(s)u_2(x,s)ds\big) \}u_2(u_1-u_2)dx.\\
\end{aligned}
\end{equation*}
By \eqref{2.2} and since $\beta \geq 1$ we derive 

\begin{equation*}
\begin{aligned}
\frac{1}{2}\frac{d}{dt} |u_1-u_2|_2^2 &+ a\int_\Om |\nabla (u_1-u_2)|^2dx  +\int_\Om  (u_1-u_2)^2dx \cr
&\leq \int_\Om |\alpha \big (\int_\Om g(x)u_1(x,t) dx \big )-  \alpha \big( \int_\Om g(x)u_2(x,t) dx \big)||\nabla u_2|| \nabla (u_1-u_2)|dx  \cr
&~~~~~~~~~~~~+  \int_\Om |\{\beta\big(\int_0^t h(s)u_1(x,s)ds\big) - \beta\big(\int_0^t h(s)u_2(x,s)ds\big) \}||u_2||(u_1-u_2)|dx,\cr
\end{aligned}
\end{equation*}
which implies
\begin{equation*}
\begin{aligned}
\frac{1}{2}\frac{d}{dt} |u_1-u_2|_2^2 &+ a\int_\Om |\nabla (u_1-u_2)|^2dx  +\int_\Om  (u_1-u_2)^2dx \cr
&\leq \int_\Om C_\alpha |\int_\Om g(x)(u_1(x,t)-u_2(x,t) )dx| |\nabla u_2|| \nabla (u_1-u_2)|dx  \cr
&~~~~~~~~~~~~ + \int_\Om C_\beta |\int_0^t h(s)(u_1(x,s) -u_2(x,s))ds||u_2||(u_1-u_2)|dx \cr
&\leq \int_\Om C_\alpha g_\infty \Big(\int_\Om |u_1(x,t)-u_2(x,t)|dx\Big) |\nabla u_2|| \nabla (u_1-u_2)|dx  \cr
&~~~~~~~~~~~~ + \int_\Om C_\beta h_\infty \int_0^t |u_1(x,s) -u_2(x,s)|ds|u_2(x,t)||(u_1-u_2)(x,t)|dx \cr
\end{aligned}
\end{equation*}
where $g_\infty$, $h_\infty$ denotes the $L^\infty(\Om)$ and $L^\infty(0,T)$ norms of $g$ and $h$. Now we use \eqref{3.1} and the young inequality 
\begin{equation*}
ab\leq \epsilon a^2 + C_\epsilon b^2
\end{equation*}
to get 
\begin{equation*}
\begin{aligned}
\frac{1}{2}\frac{d}{dt} |u_1-u_2|_2^2 &+ a\int_\Om |\nabla (u_1-u_2)|^2dx  +\int_\Om  (u_1-u_2)^2dx \cr
&\leq \int_\Om C_\epsilon\Big\{C_\alpha  g_\infty \Big(\int_\Om |u_1(x,t)-u_2(x,t)|dx\Big) |\nabla u_2|\Big\}^2 +\epsilon | \nabla (u_1-u_2)|^2dx  \cr
&~~~~~~~~~~~~+  \int_\Om C_\beta h_\infty K \frac{1}{2}\frac{d}{dt} \Big(\int_0^t |u_1(x,s) -u_2(x,s)|ds\Big)^2dx. \cr
\end{aligned}
\end{equation*}
Choosing $\epsilon=\frac{a}{2}$ it comes
\begin{equation*}
\begin{aligned}
\frac{1}{2}\frac{d}{dt} |u_1-u_2|_2^2 &+ \frac{a}{2}\int_\Om |\nabla (u_1-u_2)|^2dx  +\int_\Om  (u_1-u_2)^2dx \cr
&\leq \int_\Om C_\epsilon C_\alpha^2  g_\infty^2 |\nabla u_2|^2\Big(\int_\Om |u_1(x,t)-u_2(x,t)|dx\Big)^2 dx  \cr
&~~~~~~~~~~~~ + \int_\Om C_\beta h_\infty K \frac{1}{2}\frac{d}{dt} \Big(\int_0^t |u_1(x,s) -u_2(x,s)|ds\Big)^2dx \cr
&\leq \int_\Om C_\epsilon C_\alpha^2  g_\infty^2 |\nabla u_2|^2|\Om| |u_1-u_2|_2^2 dx  \cr
&~~~~~~~~~~~~  +\int_\Om C_\beta h_\infty K \frac{1}{2}\frac{d}{dt} \Big(\int_0^t |u_1(x,s) -u_2(x,s)|ds\Big)^2dx. \cr
\end{aligned}
\end{equation*}
We used H\"older's inequality, $|\Om|$ denotes the measure of $\Om$. Thus we obtain
 \begin{equation*}
\begin{aligned}
\frac{1}{2}\frac{d}{dt} |u_1-u_2|_2^2 &+ \frac{a}{2}\int_\Om |\nabla (u_1-u_2)|^2dx  +\int_\Om  (u_1-u_2)^2dx \cr
&\leq C_\epsilon C_\alpha^2  g_\infty^2 |\nabla u_2|_2^2|\Om| |u_1-u_2|_2^2   \cr
&~~~~~~~~~~~~ + \int_\Om C_\beta h_\infty K \frac{1}{2}\frac{d}{dt} \Big(\int_0^t |u_1(x,s) -u_2(x,s)|ds\Big)^2dx. \cr
\end{aligned}
\end{equation*}
Integrating between 0 and $t$ we derive
 \begin{equation*}
\begin{aligned}
 |u_1-u_2|_2^2 &\leq 2\int_0^t C_\epsilon C_\alpha^2  g_\infty^2 |\nabla u_2|_2^2|\Om|  |u_1-u_2|_2^2 dt  \cr
&~~~~~~~~~~~~ + \int_\Om C_\beta h_\infty K\Big(\int_0^t |u_1(x,s) -u_2(x,s)|ds\Big)^2dx \cr
&\leq 2\int_0^t C_\epsilon C_\alpha^2  g_\infty^2 |\nabla u_2|_2^2|\Om|  |u_1(x,t)-u_2(x,t)|_2^2dt  \cr
&~~~~~~~~~~~~ + \int_\Om C_\beta h_\infty Kt\int_0^t |u_1(x,s) -u_2(x,s)|^2dsdx \cr
&= \int_0^t \Big( 2C_\epsilon C_\alpha^2  g_\infty^2  |\Om| |\nabla u_2|_2^2+ C_\beta h_\infty K t \Big) |u_1-u_2|_2^2 dt.
\end{aligned}
\end{equation*}
Since $\Big(2C_\epsilon C_\alpha^2  g_\infty^2  |\Om| |\nabla u_2|_2^2+ C_\beta h_\infty K t \Big) \in L^1(0,T)$ the result follows from the Gronwall inequa\-lity.
\end{proof}
\section{Stationary problem} 
\vskip .4 cm  
In this section we consider $u$ solution to \eqref{1.1} and we assume 
 \begin{equation}\label{4.1}
 f, u_0, g, h\geq 0.
  \end{equation}
  Moreover we assume that 
   \begin{equation}\label{4.2}
\beta (z) ~~\text{admits a limit when } z\to + \infty.
  \end{equation}
First notice that \eqref{4.1} implies that $u \geq 0$. Indeed multiplying \eqref{1.1} by $- u^-$ we get 
   \begin{equation*}
\frac{1}{2} \frac{d}{dt} |u^-|_2^2 + \alpha (\int_\Om g ~u dx) \int_\Om |\nabla u^- |^2 dx +  \int_\Om \beta (\int_0^t h~uds)(u^-)^2 dx = - (f, u^-) \leq 0.
  \end{equation*}
Since $\alpha, \beta$ are positive we get 
 \begin{equation*}
 \frac{1}{2} \frac{d}{dt} |u^-|_2^2 \leq 0
  \end{equation*}
  i.e. $u^-=0$ since $u^-(x,0)=0$. Since $u\geq 0$, then 
   \begin{equation*}
t\to  \int_0^t h(s)~u(x,s)~ds
   \end{equation*}
is nondecreasing in time and has a limit when $t\to +\infty$ for almost every $x \in \Om$ and so does 
   \begin{equation*}
\beta\Big(\int_0^t h(s)~u(x,s)~ds\Big).
   \end{equation*}
We denote by $\beta_\infty(x) \in L^\infty(\Om)$ this limit. Then the stationary problem associated to \eqref{1.1} is : find $u_\infty$ weak solution to 
   \begin{equation}\label{4.3}
\begin{cases} -\alpha\big( \int_\Om g(x)u_\infty(x)dx\big)\Delta u_\infty + \beta_\infty u_\infty = f(x) \text{ in } \Om,\\
u_\infty  = 0  \text{ on } \pt \Om.
\end{cases}
\end{equation}
For convenience we set 
  \begin{equation}\label{4.4}
\ell(u) = \int_\Om g(x)u(x)dx
  \end{equation}
and for any $\mu >0$ we denote by $u_\mu$ the weak solution to 
  \begin{equation}\label{4.5}
\begin{cases} -\mu \Delta u_\mu + \beta_\infty u_\mu = f(x) \text{ in } \Om,\\
u_\mu  = 0  \text{ on } \pt \Om.
\end{cases}
\end{equation}
As usual, solving a problem like \eqref{4.3} reduces to solve an equation in $\R$ (see \cite{Ro}, \cite{Chang}). Here arguing on $\ell(u)$ or  $\alpha(\ell(u))$ offers two different equations. Indeed we have 
\begin{theorem} The mapping $u \to \ell(u)$ is a one-to-one mapping from the set of solutions to \eqref{4.3} into the set of solutions of the equation in $\R$
\begin{equation}\label{4.6}
\mu = \ell(u_{\alpha(\mu)}).
\end{equation}
\end{theorem}
\begin{proof} Suppose that $u_\infty$ is solution to \eqref{4.3}. Then, with our notation for $u_\mu$
\begin{equation*}
u_\infty = u_{\alpha(\ell(u_\infty))}
  \end{equation*}
  this implies
  \begin{equation*}
\ell(u_\infty) = \ell(u_{\alpha(\ell(u_\infty))})
  \end{equation*}
i.e.
$\ell(u_\infty)$ is solution to \eqref{4.6}. Conversely, suppose that $\mu$ is solution to \eqref{4.6}. Then, $u_{\alpha(\mu)}$ satisfies 
  \begin{equation*}
\begin{cases} -\alpha(\mu) \Delta u_{\alpha(\mu)}+ \beta_\infty u_{\alpha(\mu)}= f(x) \text{ in } \Om,\\
u_{\alpha(\mu)} = 0  \text{ on } \pt \Om.
\end{cases}
\end{equation*}
Since, by \eqref{4.6}, $\alpha(\mu) =\alpha( \ell(u_{\alpha(\mu)}))$, $u_{\alpha(\mu)}$ is solution to \eqref{4.3}. The injectivity of the map $u \to \ell(u)$ is due to the fact that if $\ell(u_1) = \ell(u_2)$ when $u_1$ and $u_2$ are solutions to \eqref{4.3} then clearly $u_1=u_2$. This completes the proof of the theorem.
\end{proof}

Similarly we have 
\begin{theorem} The mapping $u \to \alpha(\ell(u))$ is a one-to-one mapping from the set of solutions to \eqref{4.3} into the set of solutions of the equation in $\R$
\begin{equation}\label{4.7}
\mu = \alpha(\ell(u_{\mu})).
\end{equation}
\end{theorem}
\begin{proof} Suppose that $u_\infty$ is solution to \eqref{4.3}. Then, with our notation for $u_\mu$
\begin{equation*}
u_\infty = u_{\alpha(\ell(u_\infty))}
  \end{equation*}
  this implies that 
  \begin{equation*}
\alpha(\ell(u_\infty)) = \alpha(\ell(u_{\alpha(\ell(u_\infty))})
  \end{equation*}
i.e. $\alpha(\ell(u_\infty)) $ is solution to \eqref{4.7}. Conversely, suppose that $\mu$ is solution to \eqref{4.7}. Then $u_\mu$ is solution to 
  \begin{equation*}
\begin{cases} -\alpha(\ell(u_\mu)) \Delta u_\mu+ \beta_\infty u_\mu= f(x) \text{ in } \Om,\\
u_\mu= 0  \text{ on } \pt \Om,
\end{cases}
\end{equation*}
i.e. $u_\mu$ is solution to \eqref{4.3}. To prove the injectivity of the map $u \to \alpha(\ell(u))$ one has just to notice that if $ \alpha(\ell(u_1)) =  \alpha(\ell(u_2))$ when $u_1, u_2$ are solutions to \eqref{4.3} then clearly $u_1=u_2 = 
u_{\alpha(\ell(u_i))}$. This completes the proof of the theorem.
  \end{proof}

Then we can now show 
\begin{theorem} Suppose that for some constants $\alpha_0, \alpha_1$ one has  
\begin{equation}\label{4.8}
0< \alpha_0\leq \alpha\leq  \alpha_1,
\end{equation}
then the problem \eqref{4.3} admits at least one solution.
\end{theorem}
  \begin{proof} Due to \eqref{4.8} the strait line $y =\mu$ is cutting the curve $ y = \alpha(\ell(\mu))$ and the result follows from the theorem 4.2.
\end{proof}

\begin{remark} Of course \eqref{4.7} can have several solutions and even an infinity. In the case of a single solution it would be interesting and non trivial to show the convergence of $u(t)$ toward $u_\infty$.  In the next paragraph we address a simple case to show what is on stake. We made it voluntary simple in a didactic spirit.
\end{remark}

Let us suppose that $g$ is an eigenvalue of the Dirichlet problem i.e. that for some $\lambda >0$, $g$ satisfies in a weak sense

\begin{equation}\label{4.9}
-\Delta g = \lambda g \text{ in } \Om,~~ g = 0  \text{ on } \pt\Om.
\end{equation}

Then we have 
\begin{theorem} Let $g$ be solution to \eqref{4.9}. Suppose that $\beta$ is a positive constant and that the equation  
\begin{equation}\label{4.10}
(\lambda \alpha(\mu) + \beta)\mu = (f,g) >0 
\end{equation}
admits a unique solution. Then if $u(x,t)$ is solution to \eqref{1.1} and $u_\infty$ solution to \eqref{4.3} one has 
\begin{equation}\label{4.11}
|u(x,t) - u_\infty |_2 \to 0  \text{ when } t \to +\infty.
\end{equation}
\end{theorem}
  \begin{proof} It is enough to show (see \cite{CNew}) that $\ell(u(x,t)) \to \ell(u_\infty)$ when $t \to \infty$. Multiplying the equation \eqref{1.1} by $g$ and integrating on $\Omega$ one gets 
  \begin{equation*}
\frac{d}{dt}(u,g) + \alpha(\ell(u))\int_\Om \nabla u \nabla g dx + \beta(u,g) = (f,g)
\end{equation*}
i.e. using the definition of $g$ and $\ell$ it comes 
 \begin{equation*}
\frac{d}{dt}\ell(u)+ \lambda \alpha(\ell(u))\ell(u)+ \beta \ell(u) = (f,g).
\end{equation*}
Denote by $\mu_\infty$ the unique solution to \eqref{4.10}. Since we assume $(f,g) >0$ one has $\mu_\infty >0$ and 
$(\lambda \alpha(\mu) + \beta)\mu < (f,g) $ for $\mu < \mu_\infty $. Suppose that 
 \begin{equation*}
\ell(u_0) < \mu_\infty.
\end{equation*}
Since $\ell(u)$ is solution to the differential equation 
 \begin{equation*}
\frac{d}{dt}\ell(u)=  (f,g) -(\lambda \alpha(\ell(u))+ \beta) \ell(u) 
\end{equation*}
$\ell(u)$ is increasing and of course converging toward $\mu_\infty$. Similarly $\ell(u_0) >\mu_\infty$ implies that $\ell(u) $ is decreasing toward $\mu_\infty$. This completes the proof of the theorem. 
\end{proof}
\begin{remark}
In the case that we just considered one could describe the asymptotic behaviour of $u$ using the same argument when the equation \eqref{4.10} admits different isolated solutions. We leave the proof to the reader.
\end{remark}

\subsection*{Acknowledgements}
The financial support  of Gr. Roma1, IV - Mathematical Methods in NonLinear Physics National Institute for Nuclear Physics (I.N.F.N.), Rome, Italy, of \textsc{Sapienza}  University of Rome, Italy  and of the National Mathematical Physics Group (G.N.F.M.) - I.N.d.A.M., is gratefully acknowledged. \\
Thanks are due also to the PRIN 2022 project
``Mathematical Modelling of Heterogeneous Systems (MMHS)",
financed by the European Union - Next Generation EU,
CUP B53D23009360006, Project Code 2022MKB7MM, PNRR M4.C2.1.1.\\
M. Chipot thanks, Sapienza University of Rome  and Dipartimento di Scienze di Base e Applicate per l'Ingegneria,
for the kind hospitality during this research work.

\end{document}